\documentclass[journal]{IEEEtran}
\usepackage[numbers,sort&compress]{natbib}

\usepackage{algorithm}
\usepackage{algorithmic}

\usepackage{pgf}%figure
\usepackage{tikz}
\usetikzlibrary{arrows, decorations.pathmorphing, backgrounds, positioning, fit, petri, automata}
\definecolor{yellow1}{rgb}{1,0.8,0.2}

\usepackage{makecell}

\usepackage{authblk}%作者标注
\usepackage[T1]{fontenc}
\usepackage{inputenc}

\usepackage{amsmath,amsthm,amsfonts,amssymb,amsbsy,amsopn,amstext}
\usepackage{graphicx,color}
\usepackage{mathbbol,mathrsfs}
\usepackage{float,ccaption}
\usepackage{indentfirst}
\usepackage{subfigure}

\usepackage{booktabs}
\usepackage{threeparttable}
\usepackage{multirow}
\usepackage{diagbox}

\ifCLASSINFOpdf
\else
\fi
\newtheorem{thm}{Theorem}

\newtheorem{lem}{Lemma}

\newtheorem{defn}{Definition}

\newtheorem{ass}{Assumption}
\newcommand{\define}{:=}%\newcommand{\define}{\stackrel{\triangle}{=}}
\newcommand{\X}{\overline{c}}

\begin{document}
%\title{A Variance-Reduced Aggregation Based Gradient Tracking method for  Distributed Optimization with
%	Imperfect Information Sharing}

\title{A Variance-Reduced Aggregation Based Gradient Tracking method for Distributed Optimization over Directed Networks}

\author{Shengchao Zhao,~\IEEEmembership{}
	Siyuan Song,~\IEEEmembership{} 
	Yongchao Liu% <-this % stops a space
	\thanks{Shengchao Zhao and Yongchao Liu are with the School of Mathematical Sciences, Dalian University of Technology, Dalian 116024, China (e-mail: zhaoshengchao@mail.dlut.edu.cn; lyc@dlut.edu.cn).  
		
		Siyuan Song is with the
		School of Science and Engineering,
		Chinese University of Hong Kong, Shenzhen, 518172, China (e-mail: ssiyuan73@gmail.com).   }}% <-this % stops a space
\maketitle

\begin{abstract}

This paper studies the distributed optimization problem over directed networks with noisy  information-sharing. To resolve the imperfect communication  issue over directed networks, a series of  noise-robust variants  of Push-Pull/AB method have been developed. These methods improve the robustness of Push-Pull method against the information-sharing noise through adding small factors on weight matrices and replacing  the global gradient tracking with the cumulative gradient tracking. 
%However, the two techniques ensure the almost sure convergence to optimum under more strong assumptions than conventional Push-Pull method. 
Based on the two techniques, we propose a new variant of the Push-Pull method  by presenting a novel mechanism of inter-agent information aggregation, named variance-reduced aggregation (VRA). VRA helps us to release some conditions on the objective function and networks. When the objective function is convex  and the sharing-information noise is variance-unbounded, it can be shown that the proposed method converges to the optimal solution almost surely. When the objective function is strongly convex  and the sharing-information noise is variance-bounded,  the proposed method achieves the convergence rate of {\footnotesize$\mathcal{O}\left(k^{-(1-\epsilon)}\right)$} in the mean square sense, where $\epsilon$ could be close to 0 infinitely. Simulated experiments on ridge regression problems verify the effectiveness of the proposed method.

\end{abstract}

%\begin{IEEEkeywords}
%noisy  information-sharing, gradient tracking, variance-reduced aggregation, almost sure convergence, convergence rate
%\end{IEEEkeywords}

% \fi
\IEEEpeerreviewmaketitle

\section{Introduction}

\IEEEPARstart{I}{n} this paper, we consider the following distributed optimization problem
\begin{equation}\label{model}
\min_{x\in\mathbb{R}^d} f(x)\define\frac{1}{n}\sum_{j=1}^n f_j(x)
\end{equation}
over the directed networks composed by $n$ agents, where $f_j(x)$ is the local objective
function private to agent $j\in\{1,2,\cdots,n\}$. Distributed optimization problem has gained a growing interest over the past decades, motivated by various applications, ranging from large-scale machine learning \cite{boyd2011}, sensor networks \cite{Rabbat2004sensor}, parameter estimation \cite{Towfic2015}, to name a few. 
Existing algorithms for distributed optimization problems build upon local neighbor communication of different information, such as the information on the decision variable \cite{Nedic2009Dg,ram2010distributed,Johansson2008Sub},  dual variable \cite{Wei2012dadmm,LEI2016pd} and  gradient \cite{qu2017harnessing,Nedic2017achive,Xu2015Aug}.

Recently, the distributed optimization algorithms
robust to the information-sharing noise have been well studied due to the prevalence of the information-sharing noise from noisy channel \cite{Dasarathan2015robust,Kar2009imc}, commutation compression \cite{Rabbat2004sensor,Koloskova2019spare} or privacy preserving \cite{Wang2022Tailoring,Wang2017DPL}.
 To suppress the information-sharing noise, a simple but efficient approach is to add a small factor on the coupling weights, which
has been incorporated into different distributed optimization algorithms over undirected or  directed balanced networks \cite{Sri2011async,Lei2018,Zhang2019Sign,Doan2021Quanti}. Usually, a constant but small enough factor results in  biased estimates in steady-state \cite{Sri2011async}, and a decreasing factor ensures the almost sure convergence to the optimal solution \cite{Lei2018,Zhang2019Sign,Doan2021Quanti}.

To resolve the imperfect communication  issue over general directed networks, a series of  noise-robust variants  of Push-Pull/AB method \cite{Xin2018linear,pu2018push} have been developed  \cite{Wang2022Tailoring,pu2020robust,wangGT2022,Chen2022Priv,wangGT2022}. These methods combat the information-sharing noise through two techniques, i.e. adding small factors on weight matrices and replacing  the global gradient tracking with the cumulative gradient tracking. Wang and Nedic \cite{Wang2022Tailoring} redesign the conventional Push-Pull method by adding decaying factors on weight matrices for differentially-private distributed optimization. When the level set and gradient of the objective function are bounded, the method  proposed in \cite{Wang2022Tailoring}  can efficiently suppress  the differential privacy noise and acquire the almost sure convergence to the optimal solution. Noting that the dynamic-consensus
mechanism of gradient tracking would result in the information-sharing noise accumulation and thus affects the accuracy of distributed optimization,
%make the variance of global-gradient estimation grow to infinity. 
\cite{pu2020robust} suggests to track the cumulative gradient instead of the  global gradient. By combining the technique of cumulative gradient tracking with the technique of adding  small factors on  weight matrices, \cite{pu2020robust} proposes the
Robust Push-Pull method (R-Push-Pull), which is able to suppress  the general information-sharing noise and  converges to the neighborhood of the optimal solution linearly.  Chen et. al \cite{Chen2022Priv} propose a modified gradient tracking method by combining the two techniques with a state decomposition mechanism for differentially-private distributed optimization. Similar to R-Push-Pull method, the method proposed  in \cite{Chen2022Priv} is subject to steady-state errors as the factors adding on weight matrices need to be constant. Recently,  Wang and Ba\c{s}ar \cite{wangGT2022} propose a new robust gradient-tracking based method by reconstructing the cumulative gradient tracking to accommodate the decaying factors. When the communication network induced by row weight matrix is strongly connected, the authors show that the obtained solution  converges to the optimum almost surely.

Note that the methods  proposed in \cite{Wang2022Tailoring,wangGT2022} 
require more conditions on the objective function or networks to ensure the almost sure convergence to the optimum than the conventional Push-Pull method \cite{pu2018push}, such as the boundedness of the level set and gradient of the objective function \cite{Wang2022Tailoring}, the strong connectivity of the communication network induced by row weight matrix \cite{wangGT2022}. This motivates us to ask the following question:
\begin{center}
\emph{Is there a noise-robust variant of the Push-Pull method that can achieve almost sure convergence to the optimum with releasing these conditions?}
\end{center}
%under the same conditions as the conventional Push-Pull method?
In this paper, we give an affirmative answer with a variance-reduced aggregation (VRA) based gradient-tracking method. Compared with \cite{Wang2022Tailoring,wangGT2022}, the new method is designed by incorporating a novel mechanism of inter-agent information aggregation, i.e. VRA, into the cumulative gradient tracking iteration, where VRA helps us to release the conditions on the objective function and networks by reducing the  variance/bias of gradient tracking in a state-of-the-art variance-reduction fashion.
As far as we are concerned, the contributions of the paper can be summarized as follows.
\begin{itemize}
	
	\item [$\bullet$] We propose a variance-reduced aggregation based gradient tracking method for distributed optimization problem over general directed networks with noisy  information-sharing. Developed by incorporating VRA into the cumulative gradient tracking, the proposed method achieves the almost sure convergence to the optimum when the objective function is convex  and the sharing-information noise is variance-unbounded. Compared with \cite{Wang2022Tailoring},  the proposed method does not require the level set and the gradients are bounded. Compared with \cite{wangGT2022}, the proposed method does not require that the eigenvector of the row weight matrix is known or  estimated iteratively, and that the communication network induced by row weight matrix has to be strongly connected. Moreover,
    VRA is an efficient and practical mechanism to combat the information sharing noise and combining it with other distributed optimization algorithms may be of independent interest.

	\item [$\bullet$] When the objective function is strongly convex  and the information-sharing noise is variance-bounded,  all agents’ estimates converge to the same optimal solution with the convergence rate $\mathcal{O}\left(k^{-(1-\epsilon)}\right)$ in the mean square sense, where $\epsilon$ could be close to 0 infinitely. To the best of our knowledge, this is the first convergence rate result of arriving at the optimal solution, and it is a complement to the convergence rate result of arriving in the optimal solution's neighborhood \cite{Sri2011async,pu2020robust,Chen2022Priv}. We verify our theoretical results with the numerical example of ridge regression problem.
	
\end{itemize}

The rest of this paper is organized as follows. Section \ref{sec:mod-ass} presents notation and preliminary conditions on distributed optimization 
problem. Section \ref{sec:VRA} introduces the variance-reduced aggregation based gradient tracking method (VRA-GT) and its motivation. Section \ref{sec:convergence} establishes the almost sure convergence and the convergence rate in mean square sense of VRA-GT.  At last, numerical results are presented in Section \ref{sec:num-exm} to verify the effectiveness of VRA-GT.

\section{Notation and preliminary conditions}\label{sec:mod-ass}

Most of the notation and preliminary conditions refer to \cite{wangGT2022,pu2020robust}.
Denotes $\mathbb{R}^d$  as the d-dimension Euclidean space endowed with norm $\|x\|=\sqrt{\langle x,x\rangle}$ and vectors default to columns if not
otherwise specified.  $\mathbf{1}\in \mathbb{R}^{n}$, $\mathbf{0}\in \mathbb{R}^{n\times d}$ and
$\mathbf{I}\in\mathbb{R}^{n\times n}$ represent the vector of ones, the matrix of zeros and the identity matrix respectively. For any two positive sequences $\{a_k\}$ and $\{b_k\}$,  $a_k=\mathcal{O}(b_k)$ if there exists $c>0$ such that $a_k\le c b_k$. 

Consider $n$ nodes interacting over a  directed graph $\mathcal{G}=(\mathcal{V},\mathcal{E})$, where $\mathcal{V}=\{1,2,\cdots,n\}$ is the set of vertices and $\mathcal{E}\subseteq \mathcal{V}\times \mathcal{V}$ is a collection of
ordered pairs $(i, j)$ such that node $j$ can send information to node $i$. A directed graph is said to be strongly connected if there exists a directed path between any two nodes. For a nonnegative weight matrix $\mathbf{W}=\{w_{ij}\}\in\mathbb{R}^n$, we define the induced directed graph as $\mathcal{G}_{\mathbf{W}}=(\mathcal{V},\mathcal{E}_{\mathbf{W}})$ where $(i, j)\in \mathcal{E}_{\mathbf{W}}$ if and only if $w_{ij}>0$, and  denote $\mathcal{N}_{\mathbf{W},i}^{\text{in}}$ as the collection  of in-neighbors of agent $i$. Similarly, denote $\mathcal{N}_{\mathbf{W},i}^{\text{out}}$ as the set of out-neighbors of agent $i$.  We assume that $i\notin \mathcal{N}_{\mathbf{W},i}^{\text{in}}$ and $i\notin \mathcal{N}_{\mathbf{W},i}^{\text{out}}$. 

% Each agent $i\in\mathcal{V}$ holds an local copy $x_i\in \mathbb{R}^{d}$ of the decision variable and an auxiliary variables $s_i,z_i\in \mathbb{R}^{d}$. %Denote 
%\begin{align*}
%&\mathbf{x}=[x_1,x_2,\cdots,x_n]^\intercal\in \mathbb{R}^{n\times d},\\
%&\mathbf{s}=[s_1,s_2,\cdots,s_n]^\intercal\in \mathbb{R}^{n\times d},\\
%&\mathbf{z}=[z_1,z_2,\cdots,z_n]^\intercal\in \mathbb{R}^{n\times d}.
%\end{align*}
\begin{defn}\label{def:norm}
Given an arbitrary inner product $\langle x,y \rangle_W\define\langle\hat{\mathbf{W}}x,\hat{\mathbf{W}}y \rangle$ and its induced vector norm $\|x\|_W\define\|\hat{\mathbf{W}}x\|$ on $\mathbb{R}^{n}$, for any $\mathbf{x},\mathbf{y}\in \mathbb{R}^{n\times d}$, 
\begin{align*}
&\langle\mathbf{x},\mathbf{y} \rangle_W=\langle\mathbf{x}^{(1)},\mathbf{y}^{(1)} \rangle_W+\langle\mathbf{x}^{(2)},\mathbf{y}^{(2)} \rangle_W+\cdots+\langle\mathbf{x}^{(d)},\mathbf{y}^{(d)} \rangle_W
\end{align*}
and $\|\mathbf{x}\|_W=\sqrt{\langle\mathbf{x},\mathbf{x} \rangle_W}$, where  $\hat{\mathbf{W}}\in\mathbb{R}^{n\times n}$ is an invertible matrix, $\mathbf{x}^{(i)}$ and $\mathbf{y}^{(i)}$ are the $i$-th column of matrix $\mathbf{x}$ and $\mathbf{y}$ respectively.
\end{defn}

The following assumptions on the objective functions, weight matrices and their induced networks are needed.
\begin{ass}[\textbf{objective function}]\label{ass:function}
	For any $i\in\mathcal{V}$, $f_i(x)$ is convex and has $L$-Lipschitz continuous
	gradients, i.e., for any $x,y \in \mathbb{R}^d$,
	\begin{align*}
	&f_i(y)\ge f_i(x)+\left\langle \nabla f_i(x),y-x\right\rangle,\\
	&\|\nabla f_i(x)-\nabla f_i(y)\|\le  L\|x-y\|.
	\end{align*}
\end{ass}

\begin{ass}  [\textbf{weight matrices and networks}]\label{ass:matrix}
	Let $\mathcal{G}_{\mathbf{R}}=\left(\mathcal{V},\mathcal{E}_\mathbf{R}\right)$ and $\mathcal{G}_{\mathbf{C}^\intercal}=\left(\mathcal{V},\mathcal{E}_{\mathbf{C}^\intercal}\right)$ be graphs induced by  matrices $\mathbf{R}$ and $\mathbf{C}^\intercal$ respectively.  Suppose that  %where $(j,i)\in \mathcal{E}_\A$ and $(j,i)\in \mathcal{E}_{\Z^\intercal}$ only if $a_{ij}>0$ and $b_{ji}>0$ respectively.
	\begin{itemize}
		\item [(i)]The matrix $\mathbf{R}\in\mathbb{R}^{n\times n}$ is nonnegative row stochastic and $\mathbf{C}\in\mathbb{R}^{n\times n}$ is nonnegative column stochastic, i.e., $\mathbf{R}\mathbf{1}=\mathbf{1}$ and $\mathbf{1}^\intercal\mathbf{C}=\mathbf{1}^\intercal$. In addition, the diagonal entries of $\mathbf{R}$ and $\mathbf{C}$ are positive.
		\item[(ii)]  The graphs $\mathcal{G}_\mathbf{R}$ and $\mathcal{G}_{\mathbf{C}^\intercal}$ each contain at least one spanning tree. Moreover, there exists at least one node that is a root of spanning trees for both $\mathcal{G}_\mathbf{R}$ and $\mathcal{G}_{\mathbf{C}^\intercal}$, i.e. $\mathcal{R}_\mathbf{R}\cap \mathcal{R}_{\mathbf{C}^\intercal}\ne\emptyset$, where $\mathcal{R}_\mathbf{R}$ ( $\mathcal{R}_{\mathbf{C}^\intercal}$) is the set of roots of all possible spanning trees in the graph $\mathcal{G}_\mathbf{R}$ ( $\mathcal{G}_{\mathbf{C}^\intercal}$).
	\end{itemize}
\end{ass}

Assumption \ref{ass:function} is the standard condition on the objective functions, Assumption \ref{ass:matrix} is weaker than requiring that both $\mathcal{G}_\mathbf{R}$ and $\mathcal{G}_\mathbf{C}$ are strongly connected \cite{pu2018push}. 
 Under Assumption \ref{ass:matrix}, the matrix $\mathbf{R}$ has a nonnegative left eigenvector $u$ (w.r.t. eigenvalue 1) with $u^\intercal\mathbf{1}=n$, and matrix $\mathbf{C}$ has a nonnegative left eigenvector $v$ (w.r.t. eigenvalue 1) with $v^\intercal\mathbf{1}=n$. Moreover, $u^\intercal v>0$.

\section{A variance-reduced aggregation based gradient tracking method}\label{sec:VRA}
In this section, we present a variance-reduced aggregation based gradient tracking method (VRA-GT) for problem (\ref{model}) with noisy information-sharing, which is described in Algorithm \ref{alg:VRA-GT}.    %The main idea is applying a variance-reduction strategy for aggregating information on the gossip step  of iteration (\ref{GT-2}).

\begin{algorithm}[h]
	\caption{\underline{V}ariance-\underline{R}educed \underline{A}ggregation  based \underline{G}radient \underline{T}racking (VRA-GT)}\label{alg:VRA-GT}
	\begin{algorithmic}[1]
			\REQUIRE  initial values $x_{i,1},s_{i,1}=z_{i,1}\in\mathbb{R}^{d}$ for any $i\in\mathcal{V}$; step-size $\alpha_k$, positive factors $\beta_k,\gamma,\eta_k\in(0,1]$; nonnegative weight matrices $\mathbf{R}$ and $\mathbf{C}$.
		\FOR {$k=1,2,\cdots$}
		\FOR {$i=1,\cdots,n$ in parallel}
			\STATE\label{alg:s}%\textbf{Gradient accumulating:}
			$s_{i,k+1}=(1-\gamma)s_{i,k}+\gamma z_{i,k}+\nabla f_i(x_{i,k}).$
%		  {\small\begin{equation*}
%			s_{i,k+1}=(1-\gamma)s_{i,k}+\gamma z_{i,k}+\nabla f_i(x_{i,k}).
%			\end{equation*}}
		\STATE\label{alg:x}%\textbf{State estimating:}
		 agent $i$ pulls $ x_{j,k}+\xi_{j,k}$ from each $j\in \mathcal{N}_{\mathbf{R},i}^{\text{in}}$ and updates
		{\footnotesize\begin{align*}
			&x_{i,k+1}=(1-\beta_k)x_{i,k}+\beta_k\left(\sum_{j=1}^n R_{ij}x_{j,k}+\sum_{j\in\mathcal{N}_{\mathbf{R},i}^{\text{in}}} R_{ij}\xi_{j,k}\right)\\
			&\quad\quad\quad\quad-\alpha_k \left(s_{i,k+1}-s_{i,k}\right),\notag
			\end{align*}}
		where $\xi_{j,k}$ is the information-sharing noise.
		\STATE\label{alg:z}%\textbf{Variance-reduced aggregating:} 
		agent $i$ receives $C_{ij}  \frac{s_{j,k+1}-(1-\eta_k)s_{j,k}}{\eta_k}+\zeta_{j,k+1}$ pushed  from each $j\in \mathcal{N}_{\mathbf{C},i}^{\text{in}}$, and updates
		{\footnotesize\begin{align*}
			z_{i,k+1}&=\eta_k\left(\sum_{j=1}^n C_{ij}  \frac{s_{j,k+1}-(1-\eta_k)s_{j,k}}{\eta_k}+\sum_{j\in\mathcal{N}_{\mathbf{C},i}^{\text{in}}}\zeta_{j,k+1}\right)\\
			&\quad+(1-\eta_k)z_{i,k}.
			\end{align*}}
		where $\zeta_{j,k+1}$ is the information-sharing noise.
		\ENDFOR
		\ENDFOR
	\end{algorithmic}
\end{algorithm}

In Algorithm \ref{alg:VRA-GT}, $s_{i,k}$ is the tracker of cumulative gradient $\sum_{t=1}^k\sum_{j=1}^n\frac{1}{n}\nabla f_j(x)$, $x_{i,k}$ is the estimation of decision variable $x$ and $z_{i,k}$ is the tracker of  inter-agent information aggregation  $\sum_{j=1}^n C_{ij}s_{j,k}$.  The step \ref{alg:z} is named   variance-reduced aggregation (VRA), which could suppress the information-sharing noise $\zeta_{j,k}$ arising from 'push' step. It is easy to verify that Algorithm \ref{alg:VRA-GT} covers R-Push-Pull method \cite{pu2020robust} with $\eta_k=1,\beta_k=\beta$.

Denoting 
\begin{align*}
&\mathbf{s}_k= [s_{1,k},s_{2,k},\cdots,s_{n,k}]^\intercal,\\
&\mathbf{x}_k= [x_{1,k},x_{2,k},\cdots,x_{n,k}]^\intercal,\\
&\mathbf{z}_k= [z_{1,k},z_{2,k},\cdots,z_{n,k}]^\intercal,\\
&\mathbf{g}_k=\left[\nabla f_1(x_{1,k}),\nabla f_2(x_{2,k}),\cdots,\nabla f_n(x_{n,k})\right]^\intercal,\\
%\end{align*}
%\begin{align*}
&\zeta_k^C=\left[\sum_{j\in\mathcal{N}_{\mathbf{C},i}^{\text{in}}}\zeta_{j,k},\sum_{j\in\mathcal{N}_{\mathbf{C},2}^{\text{in}}}\zeta_{j,k},\cdots,\sum_{j\in\mathcal{N}_{\mathbf{C},n}^{\text{in}}}\zeta_{j,k}\right]^\intercal,\\
&\xi_k^R= \left[\sum_{j\in\mathcal{N}_{\mathbf{R},1}^{\text{in}}} R_{1j}\xi_{j,k},\sum_{j\in\mathcal{N}_{\mathbf{R},2}^{\text{in}}} R_{2j}\xi_{j,k},\cdots,\sum_{j\in\mathcal{N}_{\mathbf{R},n}^{\text{in}}} R_{nj}\xi_{j,k}\right]^\intercal,\\
&\mathbf{C}_{\gamma}=(1-\gamma)\mathbf{I}+\gamma\mathbf{C},\quad\quad\quad\quad\mathbf{R}_{k}=(1-\beta_k)\mathbf{I}+\beta_k\mathbf{R},
\end{align*}
 VRA-GT can be rewritten as a similar formula of the conventional Push-Pull scheme \cite{pu2018push}:
\begin{align}
&\mathbf{x}_{k+1}=\mathbf{R}_{k}\mathbf{x}_k+\beta_k\xi_k^R-\alpha_k\mathbf{y}_{k},\label{alg:x-1}\\
&\mathbf{y}_{k+1}=\mathbf{C}_{\gamma}\mathbf{y}_{k}+(\mathbf{g}_{k+1}+\gamma\phi_{k+1})-(\mathbf{g}_k+\gamma\phi_k),\label{alg:y}
\end{align}
where $\mathbf{y}_k=\mathbf{s}_{k+1}-\mathbf{s}_k$, $\phi_k=\mathbf{z}_k-\mathbf{C}\mathbf{s}_k$. The main difference between VRA-GT and conventional Push-Pull method is the presence of terms $\beta_k\xi_k^R$ and $\gamma\phi_k$, where $\xi_k^R$ is the information-sharing  noise at 'pull' step, $\phi_{k}$ is the  error of $z_{i,k}$ estimating the information-aggregation $\sum_{j=1}^n C_{ij}s_{j,k}$, $\beta_k$ and $\gamma$ are the factors to suppress $\xi_k^R$ and $\phi_{k}$. Note that Algorithm \ref{alg:VRA-GT} eliminates the influence of term $\beta_k\xi_k^R$ by employing decaying factor $\beta_k$ \cite{Wang2022Tailoring}, and 
eliminates the influence of term $\gamma\phi_k$ by making $\phi_k$  decay to zero
through the VRA mechanism. Then, Algorithm \ref{alg:VRA-GT} may have the convergence guarantee similar to the conventional Push-Pull method, which explains the reason why VRA-GT can release the conditions on the objective function and networks \cite{wangGT2022,Wang2022Tailoring}. To see the convergence of $\phi_k=\mathbf{z}_k-\mathbf{C}\mathbf{s}_k\longrightarrow \mathbf{0}$,  we reformulate VRA as 
\begin{align}\label{eq-z}
&\mathbf{z}_{k+1}\notag\\
%&=(1-\eta_k)\mathbf{z}_{k}+\eta_k \mathbf{C}\left(\frac{\mathbf{s}_{k+1}-(1-\eta_k)\mathbf{s}_{k}}{\eta_k}\right)+\eta_k\zeta_k\notag\\
&=(1-\eta_k)\left(\mathbf{z}_{k}+\mathbf{C}\mathbf{s}_{k+1}-\mathbf{C}\mathbf{s}_{k}\right)+\eta_k\left( \mathbf{C}\mathbf{s}_{k+1}+\zeta_{k+1}^C\right),
\end{align}
which is in the form of the hybrid variance reduced technique   and  has the theoretical guarantee of  $\mathbf{z}_{k+1}$ approximating $\mathbf{C}\mathbf{s}_{k+1}$ \cite{Ashok2019Momentum}. 
The following lemma rigorously shows the convergence of $\phi_k\longrightarrow \mathbf{0}$. 
\begin{thm}\label{thm:VRA-conv}
	Let  $\{\mathbf{s}_k\}$ be an arbitrary random variable sequence. Suppose that for any $i\in\mathcal{V}$, (a)  $\{\zeta_{i,k}\}$ is the collection of zero-mean independent random variables, (b) $\{\zeta_{i,k}, \mathbf{s}_{k}\}$ is independent for every $k$, (c) $\sum_{t=1}^\infty \eta_{t}=\infty$, $\sum_{t=1}^\infty \eta_{t}^2<\infty$, $\sum_{t=1}^\infty \eta_{t}^2 \mathbb{E}\left[\|\zeta_{i,t}\|^2\right]<\infty$. Then for the sequence $\{\mathbf{z}_k\}$ generated by (\ref{eq-z}),
	\begin{equation}\label{VRA-conv}
	\left\|\mathbf{z}_k-\mathbf{C}\mathbf{s}_k\right\|^2\longrightarrow \mathbf{0},\quad \sum_{t=1}^\infty \eta_{t} \left\|\mathbf{z}_k-\mathbf{C}\mathbf{s}_k\right\|^2<\infty
	\end{equation} 
	almost surely, where the matrix norm $\|\cdot\|$ on $\mathbb{R}^{n\times d}$ is defined by Definition \ref{def:norm} with $\hat{\mathbf{W}}=\mathbf{I}$. Moreover, if there exists a constant $\sigma$ such that 
	$$\mathbb{E}\left[\|\zeta_{i,t}\|^2\right]\le \sigma^2,~ \forall i\in\mathcal{V},~t\in \mathbb{N}$$ 
	 and 
	\begin{equation*}
	\eta_k=\left\{
	\begin{aligned}
	&\frac{\eta}{k^{a_1}},\quad\quad \eta\in(0,1),a_1\in (0.5,1),\\
	&\frac{\eta}{k+1},\quad \eta\in(1,2),
	\end{aligned}\right.
	\end{equation*}
    we have
	\begin{equation}\label{VRA-conv-1}
	\mathbb{E}\left[\left\|\mathbf{z}_k-\mathbf{C}\mathbf{s}_k\right\|^2\right]\le c_\eta\eta_k
	\end{equation}
	for some $c_\eta>0$.
\end{thm}
\begin{proof}
	By (\ref{eq-z}),
	\begin{align*}
	&\left\|\mathbf{z}_{k+1}-\mathbf{C}\mathbf{s}_{k+1}\right\|^2\\
	%&=\left\|(1-\eta_k)\left(\mathbf{z}_{k}+\mathbf{C}\mathbf{s}_{k+1}-\mathbf{C}\mathbf{s}_{k}\right)+\eta_k\left( \mathbf{C}\mathbf{s}_{k+1}+\zeta_{k+1}^C\right)\right\|^2\\
	&=\left\|(1-\eta_k)\left(\mathbf{z}_{k}-\mathbf{C}\mathbf{s}_{k}\right)+\eta_k\zeta_{k+1}^C\right\|^2\\
	&=(1-\eta_k)^2\left\|\mathbf{z}_{k}-\mathbf{C}\mathbf{s}_{k}\right\|^2+\eta_k^2\left\|\zeta_{k+1}^C\right\|^2\\
	&\quad+2\left\langle(1-\eta_k)\left(\mathbf{z}_{k}-\mathbf{C}\mathbf{s}_{k}\right),\eta_k\zeta_{k+1}^C\right\rangle.
	\end{align*}
	Define $\sigma$-algebra $\mathcal{F}_k^{'}=\sigma\{\mathbf{s}_1,\mathbf{z}_1, \zeta_2^C,\cdots,\zeta_{k}^C\}$. Noting that $\mathbb{E}\left[\zeta_{k+1}^C|\mathcal{F}_k^{'}\right]=0$, we have
	\begin{align}\label{vra-ie-1}
	&\mathbb{E}\left[\left\|\mathbf{z}_{k+1}-\mathbf{C}\mathbf{s}_{k+1}\right\|^2|\mathcal{F}_k^{'}\right]\notag\\
	&=(1-\eta_k)^2\left\|\mathbf{z}_{k}-\mathbf{C}\mathbf{s}_{k}\right\|^2+\eta_k^2\mathbb{E}\left[\left\|\zeta_k^C\right\|^2|\mathcal{F}_k^{'}\right]\notag\\
	&\le (1-\eta_k)^2\left\|\mathbf{z}_{k}-\mathbf{C}\mathbf{s}_{k}\right\|^2+\eta_k^2\sigma_{\xi,k}^2\notag\\
	&=(1+\eta_k^2-2\eta_k)\left\|\mathbf{z}_{k}-\mathbf{C}\mathbf{s}_{k}\right\|^2+\eta_k^2\sigma_{\xi,k}^2,
	\end{align}
	where $\sigma_{\xi,k}^2=n\sum_{j=1}^n\mathbb{E}\left[\left\|\zeta_{i,k}\right\|^2\right]$. Then by the Robbins-Siegmund Lemma (Lemma A.1 in Supplementary Materials Section A), $\left\|\mathbf{z}_{k}-\mathbf{C}\mathbf{s}_{k}\right\|^2$ converges to some random variable and 
	$$\sum_{t=1}^\infty \eta_{t} \left\|\mathbf{z}_{t}-\mathbf{C}\mathbf{s}_{t}\right\|^2<\infty$$
	 almost surely. Note that $\sum_{t=1}^\infty \eta_{t} =\infty$. The fact $\sum_{t=1}^\infty \eta_{t} \left\|\mathbf{z}_{t}-\mathbf{C}\mathbf{s}_{t}\right\|^2<\infty $ leads to $$\liminf_{k\rightarrow\infty} \left\|\mathbf{z}_{k+1}-\mathbf{C}\mathbf{s}_{k+1}\right\|^2=0.$$
	  The preceding relation combines the convergence of $\left\|\mathbf{z}_{k}-\mathbf{C}\mathbf{s}_{k}\right\|^2$ implies
	  $$\lim_{k\rightarrow\infty} \left\|\mathbf{z}_{k}-\mathbf{C}\mathbf{s}_{k}\right\|^2=0.$$
	  To summarize, (\ref{VRA-conv}) holds.
	
	If $\mathbb{E}\left[\|\zeta_{i,t}\|^2\right]\le \sigma^2$ and $\eta_k=\frac{\eta}{k^{a_1}}$ or $\eta_k=\frac{\eta}{k+1}$, by taking expectation on both sides of (\ref{vra-ie-1}), we have 
	\begin{align*}
	&\mathbb{E}\left[\left\|\mathbf{z}_{k+1}-\mathbf{C}\mathbf{s}_{k+1}\right\|^2\right]\\
	&\le(1-\eta_k)\mathbb{E}\left[\left\|\mathbf{z}_{k}-\mathbf{C}\mathbf{s}_{k}\right\|^2\right]+\eta_k^2n^2\sigma^2.
	\end{align*}
	 Applying \cite[Lemmas 4, 5 in Chapter 2]{polyak1987Introduction} on above inequality, we arrive at (\ref{VRA-conv-1}).
\end{proof}

Theorem \ref{thm:VRA-conv} establishes the almost sure convergence of VRA when the variance of the information-sharing noise is unbounded,  and provides its convergence rate when the information-sharing noise is variance-bounded. Theorem \ref{thm:VRA-conv} implies error $\phi_k$ may decay to zero in the almost sure and mean square senses, which is helpful for proving the convergence of Algorithm \ref{alg:VRA-GT}. On the other hand, the mechanism VRA is independent of specific distributed optimization algorithms as there are no assumptions for random variable $\mathbf{s}_{k}$ and  weight matrix $\mathbf{C}$, which provides the flexibility of combination VRA with different  distributed optimization  algorithms.

\section{Convergence Analysis of VRA-GT}\label{sec:convergence}
In this section, we study the convergence properties of VRA-GT method. The following lemma is a technical result,  which  provides some norms for studying the consensus of VRA-GT.
\begin{lem}[{\cite[Lemma 3]{Song2021CompressedGT}}]\label{lem:norm}
	Under Assumption \ref{ass:matrix},\\ (i) there exist invertible matrices $\hat{\mathbf{R}}$, $\hat{\mathbf{C}}\in\mathbb{R}^{n\times n}$ and the corresponding induced inner products
	\begin{equation*}
	\langle x, y\rangle_R\define\langle \hat{\mathbf{R}}x, \hat{\mathbf{R}}y\rangle,\quad \langle x, y\rangle_C\define\langle \hat{\mathbf{C}}x, \hat{\mathbf{C}}y\rangle
	\end{equation*}
	and vector norms
	\begin{equation*}
	\|x\|_R\define\left\|\hat{\mathbf{R}}x\right\|,\quad \|x\|_C\define\left\|\hat{\mathbf{C}}x\right\|,\quad \forall \mathbf{x}\in\mathbb{R}^{n};
	\end{equation*}
	(ii) let $\|\cdot\|_*$ and $\|\cdot\|_{**}$ be any two vector norms of $\|\cdot\|$, $\|\cdot\|_{\mathbf{R}}$ or $\|\cdot\|_{\mathbf{C}}$. There exists a constant $\X>1$ such that
	\begin{equation*}
	\|x\|_*\le \bar{c} \|x\|_{**},\quad \forall x\in\mathbb{R}^{n};
	\end{equation*}
	(iii) the corresponding matrix norms satisfy:
	\begin{equation*}
	\left\|\mathbf{C}\gamma-\frac{v\mathbf{1}^\intercal}{n}\right\|_C<1-\gamma\rho_C, \quad \left\|\mathbf{R}_k-\frac{\mathbf{1}u^\intercal}{n}\right\|_R\le 1-\beta_k\rho_R,
	\end{equation*}
	where $\rho_C,\rho_R$ are constants in $(0,1]$.
\end{lem}

For the convenience of convergence analysis,  we define an auxiliary sequence $\{\mathbf{y}_{k}^{'}\}$
\begin{equation*}
\mathbf{y}_{k+1}^{'}=\mathbf{C}_{\gamma}\mathbf{y}_{k}^{'}+\mathbf{g}_{k+1}-\mathbf{g}_{k}, \quad \mathbf{y}_{1}^{'}=\mathbf{g}_{1},
\end{equation*}
and two weighted averages $$\bar{x}_{k+1}\define\frac{u^\intercal}{n}\mathbf{x}_{k+1},~\bar{y}_{k+1}^{'}\define\frac{\mathbf{1}^\intercal}{n}\mathbf{y}_{k+1}^{'},$$ 
where  $\mathbf{y}_{k}^{'}$ stand for the trackers of global gradient based on local exact gradient information. The following lemma quantifies the upper bounds of consensus errors and optimality gap.
\begin{lem}\label{lem:couple}
	Suppose that (a) parameters $\beta_k,\gamma,\eta_k<1$, (b) Assumptions \ref{ass:function}-\ref{ass:matrix} hold. Then 
	\begin{align}\label{ie-y}
	&\mathbb{E}\left[\left\|\mathbf{y}_{k+1}^{'}-v\bar{y}_{k+1}^{'}\right\|_C^2\bigg|\mathcal{F}_k\right]\notag\\
	&\le\frac{1+\rho_\gamma^2}{2}\left\|\mathbf{y}_{k}^{'}-v \bar{y}_{k}^{'}\right\|_C^2+c_1\beta_{k}^2\|\mathbf{R}-\mathbf{I}\|^2\left\|\mathbf{x}_k-\mathbf{1}\bar{x}_k\right\|^2\notag\\
	&\quad +c_1\beta_k^2\mathbb{E}\left[\left\|\xi_k^R\right\|^2\right]+c_1\alpha_k^2\left\|\mathbf{y}_{k}\right\|^2,
	\end{align}
	\begin{align}\label{ie-x}
	&\mathbb{E}\left[\left\|\mathbf{x}_{k+1}-\mathbf{1}\bar{x}_{k+1}\right\|_R^2|\mathcal{F}_k\right]\notag\\
	&\le \left(1-\beta_k\rho_R\right)\left\|\mathbf{x}_{k}-\mathbf{1} \bar{x}_{k}\right\|_R^2+\frac{\alpha_k^2c_2}{\beta_k\rho_R}\left\|\mathbf{y}_{k}\right\|^2\notag\\
	&\quad+c_2\beta_k^2\mathbb{E}\left[\left\|\xi_k^R\right\|^2\right]~~~~~~~~~
	\end{align}
	and for any optimal solution $x^*$ of problem (\ref{model}),
	\begin{align}\label{ie-opt}
	&\mathbb{E}\left[\left\|\bar{x}_{k+1}-x^*\right\|^2|\mathcal{F}_k\right]\notag\\
	&\le \left(1+\frac{\alpha_k^2}{\beta_k}\right)\left\|\bar{x}_{k}-x^*\right\|^2+\frac{2\alpha_k^2\|u\|^2}{n^2}\left\|\mathbf{y}_{k}\right\|^2\notag\\
	&\quad+\frac{2\beta_k^2\|u\|^2}{n^2}\mathbb{E}\left[\left\|\xi_k^R\right\|^2\right]-2\alpha_k\frac{u^\intercal v}{n}\left(f(\bar{x}_k)- f(x^*)\right)\notag\\
	&\quad+\beta_k\left\|\frac{u^\intercal}{n}\mathbf{y}_{k}-\frac{u^\intercal v}{n}\nabla f(\bar{x}_k)^\intercal\right\|^2,
	\end{align}
	where $\mathcal{F}_k=\sigma\{\mathbf{x}_1,\mathbf{s}_1,\mathbf{z}_1, \zeta_2^C,\cdots,\zeta_{k}^C,\xi_1^R,\cdots,\xi_{k-1}^R\}$,
	\begin{align*}
	&\rho_\gamma=1-\gamma\rho_C,~c_1=3\frac{1+\rho_\gamma^2}{1-\rho_\gamma^2}\bar{c}^2c_v^2L^2,\\
	&c_2=\bar{c}^2\left\|\mathbf{I}-\frac{\mathbf{1}u^\intercal}{n}\right\|^2,~c_v=\left\|\mathbf{I}-\frac{v\mathbf{1}^\intercal}{n}\right\|,
	\end{align*} 
	 $\bar{c}$ is defined in Lemma \ref{lem:norm} (ii),  $\rho_C$ and $\rho_R$ are defined in Lemma \ref{lem:norm} (iii), the matrix norms $\|\cdot\|$, $\|\cdot\|_R$, $\|\cdot\|_C$ on $\mathbb{R}^{n\times d}$ are defined by Definition \ref{def:norm} with $\hat{\mathbf{W}}=\mathbf{I}$, $\hat{\mathbf{R}}$ and $\hat{\mathbf{C}}$ respectively.
\end{lem}
\begin{proof}
See Supplementary Materials Section B for the detailed proof.	
\end{proof}

The next lemma is a technical result which quantifies the errors accumulated during gradient tracking and the upper bound of the gradient tracker.

\begin{lem}\label{lem:noi-bound}
	Under the conditions of Lemma \ref{lem:couple},\\
	(i)
	\begin{equation*}
	\left\|\mathbf{y}_{k}-\mathbf{y}_{k}^{'}\right\|^2\le c_3\sum_{t=1}^{k}\rho_\gamma^{k-t}\|\omega_t\|^2,
	\end{equation*}
    (ii)
	\begin{align*}
	\left\|\mathbf{y}_{k}\right\|^2
	&\le 4c_3\sum_{t=1}^{k}\rho_\gamma^{k-t}\|\omega_t\|^2+4\left\|\mathbf{y}_{k}^{'}-v\bar{y}_{k}^{'}\right\|^2\\
	&\quad+4\frac{\|v\|^2L^2}{n}\left\|\mathbf{x}_{k}^{'}-\mathbf{1}\bar{x}_k\right\|^2+4\|v\|^2L^2\left\|\bar{x}_k-x^*\right\|^2,
	\end{align*}
	(iii)
	\begin{align*}
	&\left\|\frac{u^\intercal}{n}\mathbf{y}_{k}-\frac{u^\intercal v}{n}\nabla f(\bar{x}_k)\right\|^2\\
	&\le 3\frac{c_3\|u\|^2}{n^2}\sum_{t=1}^{k}\rho_\gamma^{k-t}\|\omega_t\|^2\\
	&\quad+3\frac{\|u\|^2}{n^2}\left\|\mathbf{y}_{k}^{'}-v\bar{y}_{k}^{'}\right\|^2+3\frac{(u^\intercal v)^2\|v\|^2L^2}{n^3}\left\|\mathbf{x}_{k}^{'}-\mathbf{1}\bar{x}_k\right\|^2,
	\end{align*}
	where  $\omega_k=\gamma(\mathbf{z}_k-\mathbf{C}\mathbf{s}_k)$, $\rho_\gamma=1-\gamma\rho_C$,  $c_3=\left(\frac{2-\rho_\gamma}{1-\rho_\gamma}\right)\max\{1,\left\|\mathbf{C}_{\gamma}-\mathbf{I}\right\|_C^2\bar{c}^2\}\frac{1}{\rho_\gamma}$, $\rho_C$ and $\bar{c}$ are defined in Lemma \ref{lem:norm} (ii) and (iii).
\end{lem}
\begin{proof}
See Supplementary Materials Section C for the detailed proof.	
\end{proof}

%Lemma $\mathbf{y}_{k}-\mathbf{y}_{k}^{'}\longrightarrow 0$

For presenting the convergence of VRA-GT method, we make the following assumption on the information-sharing noise. 

\begin{ass}[\textbf{information-sharing noise {\cite{wangGT2022}}}]\label{ass:noise}
	For every $i\in\mathcal{V}$, the noise sequence $\{\xi_{i,k}\}$ and $\{\zeta_{i,k}\}$ are zero-mean independent random variables, and independent of $\mathbf{s}_1,\mathbf{x}_1$ and $\mathbf{z}_1$. Moreover,  $\sum_{t=1}^\infty \eta_{t}^2 \mathbb{E}\left[\|\zeta_{i,t}\|^2\right]<\infty$,  $\sum_{t=1}^\infty \beta_{t}^2 \mathbb{E}\left[\|\xi_{i,t}\|^2\right]<\infty$.
\end{ass}

We are ready to present the almost sure convergence of VRA-GT method.
\begin{thm}\label{thm:VRA-GT conv}
Suppose that (a) parameters $\beta_k,\gamma,\eta_k<1$ and $\{\beta_k\}$ is not increase,  $\sum_{t=1}^\infty \beta_{t}=\infty$, $\sum_{t=1}^\infty \beta_{t}^2<\infty$, $\sum_{t=1}^\infty \eta_{t}=\infty$, $\sum_{t=1}^\infty \eta_{t}^2<\infty$, (b) stepsize $\sum_{t=1}^\infty \alpha_{t}=\infty$, $\sum_{t=1}^\infty \frac{\alpha_{t}^2}{\beta_{t}}<\infty$, $\lim_{k\rightarrow\infty}\frac{\alpha_k}{\beta_k}=0$, (c) Assumptions \ref{ass:function}-\ref{ass:noise} hold. Then for all $i\in\mathcal{V}$, $x_{i,t}$  converges to a same random point in $X^*\define\{x\in\mathbb{R}^d:f(x)=\min_x f(x)\}$ almost surely.	
\end{thm}
\begin{proof}
	We employ Lemma A.1 in Supplementary Materials Section A to prove the convergence of $x_{i,t}$. To this end, we rewrite the recursions (\ref{ie-y})-(\ref{ie-opt}) in the form of (A.1) in Lemma A.1. Denote 
	{\small$$v_{k+1}=\left\|\mathbf{y}_{k+1}^{'}-v\bar{y}_{k+1}^{'}\right\|_C^2+\left\|\mathbf{x}_{k+1}-\mathbf{1}\bar{x}_{k+1}\right\|_R^2+c_*\left\|\bar{x}_{k+1}-x^*\right\|^2,$$}
	where $c_*$ is some positive constant to be determined later. By Lemma \ref{lem:couple},
	\begin{align*}
	&\mathbb{E}\left[v_{k+1}|\mathcal{F}_k\right]\\
	&\le \frac{1+\rho_\gamma^2}{2}\left\|\mathbf{y}_{k}^{'}-v\bar{y}_{k}^{'}\right\|_C^2\\
	&\quad+\left(1-\beta_k\rho_R+c_1\|\mathbf{R}-\mathbf{I}\|^2\bar{c}^2\beta_{k}^2\right)\left\|\mathbf{x}_{k}-\mathbf{1} \bar{x}_{k}\right\|_R^2\\
	&\quad +\left(1+\frac{\alpha_k^2}{\beta_k}\right)c_*\left\|\bar{x}_{k}-x^*\right\|^2-2\alpha_k\frac{u^\intercal v}{n}c_*\left(f(\bar{x}_k)- f(x^*)\right)\\
	&\quad+\left(c_1+c_2+\frac{2\|u\|^2}{n^2}c_*\right)\beta_k^2\mathbb{E}\left[\left\|\xi_k^R\right\|^2\right]\\
	&\quad+c_*\beta_k\left\|\frac{u^\intercal}{n}\mathbf{y}_{k}-\frac{u^\intercal v}{n}\nabla f(\bar{x}_k)\right\|^2\\
	&\quad+\left(c_1\alpha_k^2+\frac{\alpha_k^2c_2}{\beta_k\rho_R}+\frac{2\alpha_k^2\|u\|^2}{n^2}c_*\right)\left\|\mathbf{y}_{k}\right\|^2.
	\end{align*}
	By  Lemma \ref{lem:noi-bound} (ii) and (iii) and 
%	\begin{align*}
%	&\mathbb{E}\left[v_{k+1}|\mathcal{F}_k\right]\\
%	&\le \left(\frac{1+\rho_\gamma^2}{2}+3c_*\frac{\|u\|^2}{n^2}\beta_k+4\left(c_1\alpha_k^2+\frac{\alpha_k^2c_2}{\beta_k\rho_R}+\frac{2\alpha_k^2\|u\|^2}{n^2}c_*\right)\right)\left\|\mathbf{y}_{k}^{'}-v\bar{y}_{k}^{'}\right\|_C^2\\
%	&\quad+\left(1-\beta_k\rho_R+c_1\|\mathbf{R}-\mathbf{I}\|^2\bar{c}^2\beta_{k}^2+3c_*\frac{(u^\intercal v)^2\|v\|^2L^2}{n^3}\beta_k\right)\left\|\mathbf{x}_{k}-\mathbf{1} \bar{x}_{k}\right\|_R^2\\
%	&\quad+4\frac{\|v\|^2L^2}{n}\left(c_1\alpha_k^2+\frac{\alpha_k^2c_2}{\beta_k\rho_R}+\frac{2\alpha_k^2\|u\|^2}{n^2}c_*\right)\left\|\mathbf{x}_{k}-\mathbf{1} \bar{x}_{k}\right\|_R^2\\
%	&\quad +\left(1+\frac{\alpha_k^2}{\beta_k}+4\left(c_1\alpha_k^2+\frac{\alpha_k^2c_2}{\beta_k\rho_R}+\frac{2\alpha_k^2\|u\|^2}{n^2}c_*\right)\|v\|^2L^2c_*^{-1}\right)c_*\left\|\bar{x}_{k}-x^*\right\|^2\\
%	&\quad-2\alpha_k\frac{u^\intercal v}{n}c_*\left(f(\bar{x}_k)- f(x^*)\right)+\left(c_1+c_2+\frac{2\|u\|^2}{n^2}c_*\right)\beta_k^2\left\|\mathbf{R}_{\rm d}\right\|^2\mathbb{E}\left[\left\|\zeta_k\right\|^2\right]\\
%	&\quad +\left(3c_*\frac{c_3\|u\|^2}{n^2}\beta_k+4c_3\left(c_1\alpha_k^2+\frac{\alpha_k^2c_2}{\beta_k\rho_R}+\frac{2\alpha_k^2\|u\|^2}{n^2}c_*\right)\right)\sum_{t=1}^{k}\rho_\gamma^{k-t}\|\omega_t\|^2
%	\end{align*}
	denoting
	{\small\begin{align*}
	&p_k\\
	&=\min\left\{\frac{1-\rho_\gamma^2}{2}-3c_*\frac{\|u\|^2}{n^2}\beta_k, \left(\rho_R-3c_*\frac{(u^\intercal v)^2\|v\|^2L^2}{n^3}\right)\beta_k,\right.\\
	&\quad\quad\quad\left.2\alpha_k\frac{u^\intercal v}{n}c_*\right\},\\
	&q_k\\
	&=\max \left\{4\left(c_1\alpha_k^2+\frac{\alpha_k^2c_2}{\beta_k\rho_R}+\frac{2\alpha_k^2\|u\|^2}{n^2}c_*\right),c_1\|\mathbf{R}-\mathbf{I}\|^2\bar{c}^2\beta_{k}^2\right.\\
	&\quad\quad\quad +4\frac{\|v\|^2L^2}{n}\left(c_1\alpha_k^2+\frac{\alpha_k^2c_2}{\beta_k\rho_R}+\frac{2\alpha_k^2\|u\|^2}{n^2}c_*\right),\\
	&\quad\quad\quad\left. \frac{\alpha_k^2}{\beta_k}+4\left(c_1\alpha_k^2+\frac{\alpha_k^2c_2}{\beta_k\rho_R}+\frac{2\alpha_k^2\|u\|^2}{n^2}c_*\right)\|v\|^2L^2c_*^{-1}\right\},
	\end{align*}}
	we have
	\begin{align}\label{ie-3}
	&\mathbb{E}\left[v_{k+1}|\mathcal{F}_k\right]\notag\\
	&\le(1+q_k)v_k-p_k\left(\left\|\mathbf{y}_{k}^{'}-v\bar{y}_{k}^{'}\right\|_C^2+\left\|\mathbf{x}_{k}-\mathbf{1}\bar{x}_{k}\right\|_R^2\right.\notag\\
	&\quad+f(\bar{x}_k)- f(x^*)\bigg)\notag\\
	&\quad+\left(c_1+c_2+\frac{2\|u\|^2}{n^2}c_*\right)\beta_k^2\mathbb{E}\left[\left\|\xi_k^R\right\|^2\right]\notag\\
	&\quad +\left(4c_3\left(c_1\alpha_k^2+\frac{\alpha_k^2c_2}{\beta_k\rho_R}+\frac{2\alpha_k^2\|u\|^2}{n^2}c_*\right)\right.\notag\\
	&\quad\left.+3c_*\frac{c_3\|u\|^2}{n^2}\beta_k\right)\sum_{t=1}^{k}\rho_\gamma^{k-t}\|\omega_t\|^2.
	\end{align}
	Obviously, the above inequality falls in the form of (A.1) in Lemma A.1. 
	
	Next, we verify the conditions of Lemma A.1. 
	
	By the definitions of $\alpha_k$ and $\beta_k$, $q_k$ is summable namely $\sum_{k=1}^\infty q_k<\infty$.
	
	Given $\alpha_k$ and $\beta_k$, we may choose $c_*$ such that $p_k>0$ for all $k\ge 1$, which implies the second term on the right hand side of (\ref{ie-3}) is nonnegative.  
  
  By Assumption \ref{ass:noise}, 
	\begin{equation*}
	\sum_{k=1}^\infty\left(c_1+c_2+\frac{2\|u\|^2}{n^2}c_*\right)\beta_k^2\mathbb{E}\left[\left\|\xi_k^R\right\|^2\right]<\infty.
	\end{equation*}
	In addition, {\small$$\left(4c_3\left(c_1\alpha_k^2+\frac{\alpha_k^2c_2}{\beta_k\rho_R}+\frac{2\alpha_k^2\|u\|^2}{n^2}c_*\right)+3c_*\frac{c_3\|u\|^2}{n^2}\beta_k\right)=\mathcal{O}\left(\beta_{k}\right)$$}
	and
	\begin{align*}
	&\sum_{k=1}^\infty\beta_k\sum_{t=1}^{k}\rho_\gamma^{k-t}\|\omega_t\|^2\\
	&=\lim_{K\rightarrow\infty}\sum_{k=1}^K\beta_k\sum_{t=1}^{k}\rho_\gamma^{k-t}\|\omega_t\|^2\\
	&=\lim_{K\rightarrow\infty}\sum_{k=1}^K\left(\sum_{t=k}^K\beta_t\rho_\gamma^{t-k}\right)\|\omega_k\|^2\\
	&\le\lim_{K\rightarrow\infty}\sum_{k=1}^K\frac{\beta_k}{1-\rho_\gamma}\|\omega_k\|^2=\sum_{k=1}^\infty\frac{\beta_t}{1-\rho_\gamma}\|\omega_k\|^2<\infty,
	\end{align*}
	where the first inequality holds as $\beta_{k}$ is not increase. Then the last two terms on the right hand side of (\ref{ie-3}) are summable.
	
	Summarizing above results,  the conditions of Lemma A.1 hold and thus  $v_k$ converges to some finite random variable $v^\infty$ and 
	{\small\begin{align*}
	&\sum_{k=1}^\infty p_k\left(\left\|\mathbf{y}_{k}^{'}-v\bar{y}_{k}^{'}\right\|_C^2+\left\|\mathbf{x}_{k}-\mathbf{1}\bar{x}_{k}\right\|_R^2+f(\bar{x}_k)- f(x^*)\right)<\infty
	\end{align*}}
	almost surely.  Noting that $\sum_{k=1}^\infty p_k=\infty$,  for an arbitrary sample trajectory, there exists a sub-sequence $\{k_0\}\subseteq \mathbb{N}$ such that 
	{\small\begin{align*}
	&\lim_{k_0\rightarrow\infty}\left\|\mathbf{y}_{k_0}^{'}-v\bar{y}_{k_0}^{'}\right\|_C^2+\left\|\mathbf{x}_{k_0}-\mathbf{1}\bar{x}_{k_0}\right\|_R^2+f(\bar{x}_{k_0})- f(x^*)=0.
	\end{align*}}
	Then by the convergence of $v_k$,
	\begin{equation}\label{x-conv}
	\lim_{k_0\rightarrow\infty} c_*\left\|\bar{x}_{k_0}-x^*\right\|^2=\lim_{k_0\rightarrow\infty} v_{k_0}=v^\infty.
	\end{equation}
	By the fact $\lim_{k_0\rightarrow\infty}f(\bar{x}_{k_0})- f(x^*)=0$, the continuity of $f(x)$ and the boundness of $\{\bar{x}_{k_0}\}$, 
	there exists a subsequence of $\{\bar{x}_{k_0}\}$ converge to a point $\hat{x}$ in $X^*$.  Since the choice of $x^*$
	is arbitrary, we take $x^*=\hat{x}$ in $v_k$. Then by (\ref{x-conv}),  $v^\infty=0$  on this sample trajectory. Therefore, 
	\begin{align*}
	\lim_{k\rightarrow\infty} \left\|\bar{x}_{k}-\hat{x}\right\|^2=0,\quad \lim_{k\rightarrow\infty} \left\|\mathbf{x}_{k+1}-\mathbf{1}\bar{x}_{k+1}\right\|_R^2=0,
	\end{align*}
	almost surely, and hence the desired result follows.
\end{proof}
Theorem \ref{thm:VRA-GT conv} shows that all agents can converge to a same optimal solution almost surely. Compared with \cite{Wang2022Tailoring},  the proposed method does not requires the level set and gradient are bounded. Compared with \cite{wangGT2022}, the proposed method does not require that the eigenvector of the row weight matrix is known or  estimated iteratively, and that the communication network induced by row weight matrix has to be strongly connected.

\begin{figure*}[htb]
	\centering
	\subfigure[$\sigma^2_\xi=\sigma^2_\zeta=1$]{
		\includegraphics[width=2.45in]{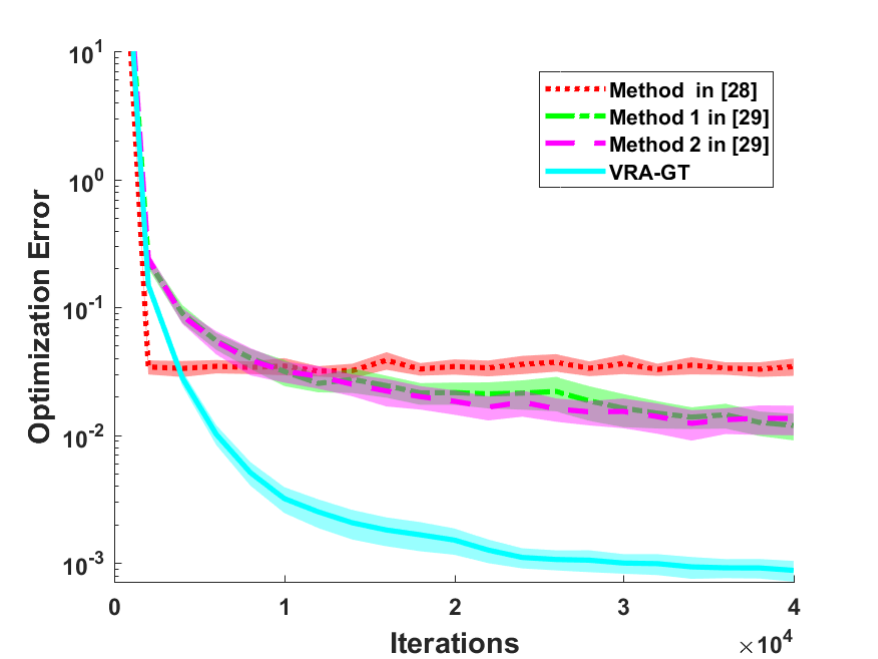}
	}\hspace{-9mm}
	\subfigure[$\sigma^2_\xi=\sigma^2_\zeta=25$]{
		\includegraphics[width=2.45in]{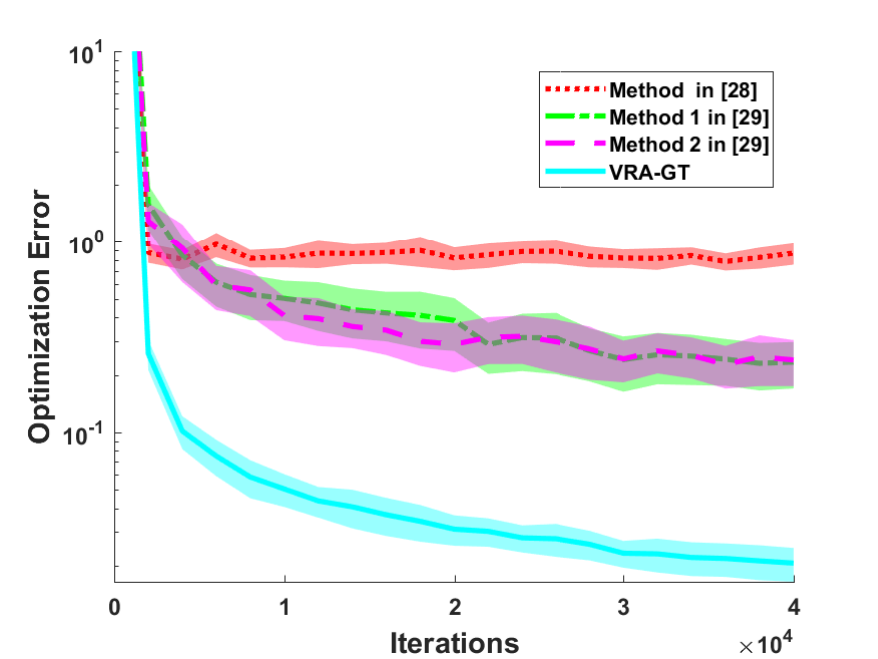}
	}\hspace{-9mm}
	\subfigure[$\sigma^2_\xi=\sigma^2_\zeta=50$]{
		\includegraphics[width=2.45in]{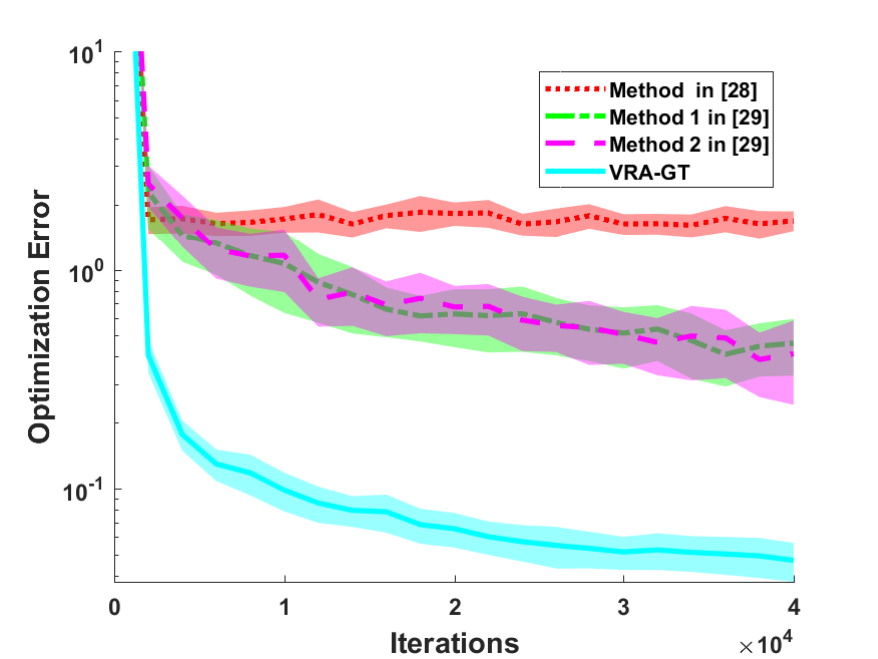}
	}
	\caption{{\small Evolutions of $\sum_{j=1}^n \left\|x_{i,k}-x^*\right\|^2$  w.r.t to the number of iterations.}}
	\label{fig-1}
\end{figure*}
The following theorem establishes the convergence rate of VRA-GT method in the mean square sense.
\begin{thm}\label{thm:VRA-GT conv-1}
Suppose that (a) Assumptions \ref{ass:function}-\ref{ass:matrix} hold, (b) $f(x)$ is $\mu$-strongly convex and there exists a constant $\sigma$ such that $\mathbb{E}\left[\|\zeta_{i,t}\|^2\right]\le \sigma^2$, $\mathbb{E}\left[\|\xi_{i,t}\|^2\right]\le \sigma^2$, (c) $\gamma<1$, $\eta_k=\frac{a_1}{k^\eta}$, $\beta_k=\frac{a_2}{k^\beta}$, $\alpha_k=\frac{a_3}{k^\alpha}$, where $a_1,a_2,a_3\in (0,1]$, $\eta,\alpha,\beta\in(0.5,1)$ and $\alpha,\beta$ satisfy $\alpha>\frac{1+\beta}{2}$.
Then
\begin{equation}\label{ie-4}
\mathbb{E}\left[v_{k+1}\right]\le\mathcal{O}\left(\frac{1}{k^{\min\{2\beta-\alpha,\beta+\eta-\alpha\}}}\right),
\end{equation}
where $v_{k+1}=\left\|\mathbf{y}_{k+1}^{'}-v\bar{y}_{k+1}^{'}\right\|_C^2+\left\|\mathbf{x}_{k+1}-\mathbf{1}\bar{x}_{k+1}\right\|_R^2+c_*\left\|\bar{x}_{k+1}-x^*\right\|^2$.
\end{thm}
\begin{proof}
Under the conditions (b) and (c), Theorem \ref{thm:VRA-conv} implies
$$\mathbb{E}\left[\|\omega_k\|^2\right]\le c_\eta\eta_t$$
for some $c_\eta>0$.
Taking exception on both sides of (\ref{ie-3}) and substituting above relation into it,
\begin{align*}
&\mathbb{E}\left[v_{k+1}\right]\notag\\
&\le(1+q_k)\mathbb{E}\left[v_k\right]-p_k\left(\left\|\mathbf{y}_{k}^{'}-v\bar{y}_{k}^{'}\right\|_C^2+\left\|\mathbf{x}_{k}-\mathbf{1}\bar{x}_{k}\right\|_R^2\right.\notag\\
&\quad+f(\bar{x}_k)- f(x^*)\bigg)\notag\\
&\quad+\left(c_1+c_2+\frac{2\|u\|^2}{n^2}c_*\right)\beta_k^2\mathbb{E}\left[\left\|\xi_k^R\right\|^2\right]\notag\\
&\quad +\left(4c_3\left(c_1\alpha_k^2+\frac{\alpha_k^2c_2}{\beta_k\rho_R}+\frac{2\alpha_k^2\|u\|^2}{n^2}c_*\right)\right.\notag\\
&\quad\left.+3c_*\frac{c_3\|u\|^2}{n^2}\beta_k\right)\sum_{t=1}^{k}\rho_\gamma^{k-t}c_\eta\eta_t\notag\\
&\le(1+q_k)\mathbb{E}\left[v_k\right]-p_k\left(\left\|\mathbf{y}_{k}^{'}-v\bar{y}_{k}^{'}\right\|_C^2+\left\|\mathbf{x}_{k}-\mathbf{1}\bar{x}_{k}\right\|_R^2\right.\notag\\
&\quad+\mu\mathbb{E}\left[\left\|\bar{x}_{k+1}-x^*\right\|^2\right]\bigg)\notag\\
&\quad+\left(c_1+c_2+\frac{2\|u\|^2}{n^2}c_*\right)\beta_k^2n\sigma^2\notag\\
&\quad +\left(4c_3\left(c_1\alpha_k^2+\frac{\alpha_k^2c_2}{\beta_k\rho_R}+\frac{2\alpha_k^2\|u\|^2}{n^2}c_*\right)\right.\notag\\
&\quad\left.+3c_*\frac{c_3\|u\|^2}{n^2}\beta_k\right)\mathcal{O}\left(\eta_k\right)\notag\\
&\le\left(1+q_k-p_k\min\left\{1,\frac{\mu}{c_*}\right\}\right)\mathbb{E}\left[v_k\right]\notag\\
&\quad+\left(c_1+c_2+\frac{2\|u\|^2}{n^2}c_*\right)\beta_k^2n\sigma^2\notag\\
&\quad +\left(4c_3\left(c_1\alpha_k^2+\frac{\alpha_k^2c_2}{\beta_k\rho_R}+\frac{2\alpha_k^2\|u\|^2}{n^2}c_*\right)\right.\notag\\
&\quad\left.+3c_*\frac{c_3\|u\|^2}{n^2}\beta_k\right)\mathcal{O}\left(\eta_k\right),
\end{align*}
where the second inequality follows from the strong convexity of $f(x)$, the fact $\mathbb{E}\left[\|\xi_{i,t}\|^2\right]\le \sigma^2$ and the relation $\sum_{t=1}^{k}\rho_\gamma^{k-t}c_\eta\eta_t\le \mathcal{O}\left(\eta_k\right)$ (\cite[Lemma 3 in Appendix A]{zhao2020asymptotic}). By condition (c), 
\begin{equation*}
1+q_k-p_k\min\left\{1,\frac{\mu}{c_*}\right\}\le 1-\mathcal{O}\left(\frac{1}{k^\alpha}\right)
\end{equation*}
and then
\begin{align*}
&\mathbb{E}\left[v_{k+1}\right]\notag\\
&\le\left(1-\mathcal{O}\left(\frac{1}{k^\alpha}\right)\right)\mathbb{E}\left[v_k\right]+\mathcal{O}\left(\frac{1}{k^{2\beta}}+\frac{1}{k^{\beta+\eta}}\right).
\end{align*}
Applying \cite[Lemma 5 in Chapter 2]{polyak1987Introduction}  on above relation, we arrive (\ref{ie-4}).
\end{proof}

Theorem \ref{thm:VRA-GT conv-1} establishes the convergence rate of VRA-GT method when the variance of the information-sharing noise is bounded and the objective function is strongly convex. Particularly, setting $\eta=\beta=1-5/8\epsilon$ and  $\alpha=1-0.25\epsilon$,  the convergence rate of VRA-GT is  $\mathcal{O}\left(\frac{1}{k^{1-\epsilon}}\right)$, where $\epsilon$ can close to zero infinitely. Moreover, Theorem \ref{thm:VRA-GT conv-1} may complement the convergence rate result of arriving in the optimal solution's neighborhood \cite{Sri2011async,pu2020robust,Chen2022Priv}.  

\section{Experimental Results}\label{sec:num-exm}

In this section, we perform a simulation study to illustrate our theoretic findings on the convergence properties of VRA-GT method.  Consider the ridge regression problem \cite{wangGT2022}:
\begin{equation}\label{sim-pro}
\min_{x\in\mathbb{R}^d} ~f(x)=\sum_{j=1}^n\left\|M_j x-v_j\right\|^2+r\|x\|^2,\\
\end{equation}
where $f_j(x)\define\left\|w_j^\intercal x-v_j\right\|^2+r\|x\|^2$ is the objective function of agent $j$, $M_j\in \mathbb{R}^{d_1\times d}$ is  the measurement matrix, $v_j\in \mathbb{R}^{d_1}$  is a noisy measurement, $r$ is the regularization parameter.
%where $\gamma=1$ is a penalty parameter.
In problem (\ref{sim-pro}),
each agent $ i\in\mathcal{V}$ has access to sample $(M_i,v_i)$ given by the linear model
$v_i=M_i^\intercal \tilde{x}+\nu_i,$
where  $\nu_i $ is the measurement noise and $\tilde{x}$ is  an unknown parameter. 

In this experiment, the settings of predetermined parameters and network topology follow from \cite{wangGT2022}. We make $r=0.05$, $d_1=3$, $d=2$, $M_i$ is generated from a uniform distribution in the unit $\mathbb{R}^{d_1\times d}$ space, $\nu_i$ follows an i.i.d. Gaussian process
with zero mean and unit variance, $\tilde{x}$ is evenly located in $[1,~10]^d$ for $\forall i\in\mathcal{V}$.  The directed graph $\mathcal{G}$ made up of 100 agents is generated by adding random links to a ring network, where a directed link exists between any two nonadjacent nodes with a probability $p=0.3$. 
For $\forall i\in \mathcal{V}$,
$\mathcal{G}_\mathbf{R}=\mathcal{G}_\mathbf{C}=\mathcal{G}$ and
\begin{equation*}
\begin{aligned}
&\mathbf{R}_{ij}=\left\{
\begin{aligned}
&\frac{1}{|\mathcal{N}_{\mathbf{R},i}^{\text{in}}|+1},\quad j\in \mathcal{N}_{\mathbf{R},i}^{\text{in}},\\
&1-\sum_{j\in\mathcal{N}_{\mathbf{R},i}^{\text{in}}}\mathbf{R}_{ij},\quad j=i,
\end{aligned}\right.
\\
&\mathbf{C}_{ji}=\left\{
\begin{aligned}
&\frac{1}{|\mathcal{N}_{\mathbf{C},i}^{\text{out}}|+1},\quad j\in\mathcal{N}_{\mathbf{C},i}^{\text{out}},\\
&1-\sum_{j\in\mathcal{N}_{\mathbf{C},i}^{\text{out}}}\mathbf{C}_{ji},\quad j=i,
\end{aligned}\right.
\end{aligned}
\end{equation*}
where $|\mathcal{N}_{\mathbf{R},i}^{\text{in}}|$ and $|\mathcal{N}_{\mathbf{C},i}^{\text{out}}|$ are the  cardinality of $\mathcal{N}_{\mathbf{R},i}^{\text{in}}$ and $\mathcal{N}_{\mathbf{C},i}^{\text{out}}$.

We run VRA-GT and the algorithms proposed in \cite{pu2020robust,wangGT2022} for 100 times and calculate the average as well as the variance of the optimization error $\sum_{j=1}^n \left\|x_{i,k}-x^*\right\|^2$ as a function of the iteration index $k$. We set  $\gamma=0.8$, $\beta_{k}=\frac{0.1}{1+k^{0.6}}$ and $\alpha_k=\frac{0.1}{1+k^{0.9}}$ for VRA-GT, $\gamma=0.5$, $\eta=0.01$ and $\alpha=0.01$ for the algorithms proposed in \cite{pu2020robust} (i.e. Robust push pull method),  $\gamma_k=\frac{1}{1+k^{0.7}}$ and $\lambda_k=\frac{1}{1+0.7k^{0.9}}$ for the methods proposed in \cite{wangGT2022}.  Their performance under Gaussian sharing-information noise with variance $\sigma^2_\xi=\sigma^2_\zeta=1,25,50$ are depicted in Figure \ref{fig-1}, where the solid curve, dot curve, dash-dot carve and dashed curve display the evaluation of VAR-GT method and the methods proposed in \cite{pu2020robust,wangGT2022} respectively.

In summary, the displayed algorithms in Figure \ref{fig-1} are all robust to the information-sharing noise with different variance (i.e. 1, 25, 50) and can converge to optimal solution with different accuracy. In the initial stage of iteration,  R-Push-Pull method has more faster convergence rate due to the stepsize and factors added on coupling weight being constant, especially for the noise with smaller variance as shown in Figure \ref{fig-1} (a). With the iterations increasing, VRA-GT and the proposed methods in \cite{wangGT2022} have preferable optimization accuracy than R-Push-Pull method for the noise with different variance levels, which can be attributed to the noise suppressing effect of decreasing factors. Obviously,  VRA-GT has the best convergence performance among the displayed methods in Figure \ref{fig-1} since VRA provides more accurate gradient-tracking result by reducing the gradient-estimation noise variance.
%{RL_d_grad_e_6-eps-converted-to}

\section*{Acknowledgment}

The authors thank Professor Yongqiang Wang for the discussions on the proof of Theorem 2. The research is supported by National Key R$\&$D Program of China No. 2022YFA1004000, the NSFC \#11971090 and  Fundamental Research Funds for the Central Universities   DUT22LAB301.
% or

\bibliographystyle{IEEEtran}
\bibliography{mybib}

\begin{thebibliography}{10}
\providecommand{\url}[1]{#1}
\csname url@samestyle\endcsname
\providecommand{\newblock}{\relax}
\providecommand{\bibinfo}[2]{#2}
\providecommand{\BIBentrySTDinterwordspacing}{\spaceskip=0pt\relax}
\providecommand{\BIBentryALTinterwordstretchfactor}{4}
\providecommand{\BIBentryALTinterwordspacing}{\spaceskip=\fontdimen2\font plus
\BIBentryALTinterwordstretchfactor\fontdimen3\font minus
  \fontdimen4\font\relax}
\providecommand{\BIBforeignlanguage}[2]{{%
\expandafter\ifx\csname l@#1\endcsname\relax
\typeout{** WARNING: IEEEtran.bst: No hyphenation pattern has been}%
\typeout{** loaded for the language `#1'. Using the pattern for}%
\typeout{** the default language instead.}%
\else
\language=\csname l@#1\endcsname
\fi
#2}}
\providecommand{\BIBdecl}{\relax}
\BIBdecl

\bibitem{boyd2011}
S.~Boyd, N.~Parikh, E.~Chu, B.~Peleato, and J.~Eckstein, \emph{Distributed
  Optimization and Statistical Learning via the Alternating Direction Method of
  Multipliers}, 2011.

\bibitem{Rabbat2004sensor}
M.~Rabbat and R.~Nowak, ``Distributed optimization in sensor networks,'' in
  \emph{Third International Symposium on Information Processing in Sensor
  Networks}, 2004, pp. 20--27.

\bibitem{Towfic2015}
Z.~J. Towfic and A.~H. Sayed, ``Stability and performance limits of adaptive
  primal-dual networks,'' \emph{IEEE Transactions on Signal Processing},
  vol.~63, no.~11, pp. 2888--2903, 2015.

\bibitem{Nedic2009Dg}
A.~Nedic and A.~Ozdaglar, ``Distributed subgradient methods for multi-agent
  optimization,'' \emph{IEEE Transactions on Automatic Control}, vol.~54,
  no.~1, pp. 48--61, 2009.

\bibitem{ram2010distributed}
S.~S. Ram, A.~Nedi{\'c}, and V.~V. Veeravalli, ``Distributed stochastic
  subgradient projection algorithms for convex optimization,'' \emph{Journal of
  Optimization Theory and Applications}, vol. 147, no.~3, pp. 516--545, 2010.

\bibitem{Johansson2008Sub}
B.~Johansson, T.~Keviczky, M.~Johansson, and K.~H. Johansson, ``Subgradient
  methods and consensus algorithms for solving convex optimization problems,''
  in \emph{2008 47th IEEE Conference on Decision and Control}, 2008, pp.
  4185--4190.

\bibitem{Wei2012dadmm}
E.~Wei and A.~Ozdaglar, ``Distributed alternating direction method of
  multipliers,'' in \emph{2012 IEEE 51st IEEE Conference on Decision and
  Control (CDC)}, 2012, pp. 5445--5450.

\bibitem{LEI2016pd}
J.~Lei, H.-F. Chen, and H.-T. Fang, ``Primal–dual algorithm for distributed
  constrained optimization,'' \emph{Systems \& Control Letters}, vol.~96, pp.
  110--117, 2016.

\bibitem{qu2017harnessing}
G.~Qu and N.~Li, ``Harnessing smoothness to accelerate distributed
  optimization,'' \emph{IEEE Transactions on Control of Network Systems},
  vol.~5, no.~3, pp. 1245--1260, 2018.

\bibitem{Nedic2017achive}
A.~Nedić, A.~Olshevsky, and W.~Shi, ``Achieving geometric convergence for
  distributed optimization over time-varying graphs,'' \emph{SIAM Journal on
  Optimization}, vol.~27, no.~4, pp. 2597--2633, 2017.

\bibitem{Xu2015Aug}
J.~Xu, S.~Zhu, Y.~C. Soh, and L.~Xie, ``Augmented distributed gradient methods
  for multi-agent optimization under uncoordinated constant stepsizes,'' in
  \emph{2015 54th IEEE Conference on Decision and Control (CDC)}, 2015, pp.
  2055--2060.

\bibitem{Dasarathan2015robust}
S.~Dasarathan, C.~Tepedelenlioğlu, M.~K. Banavar, and A.~Spanias, ``Robust
  consensus in the presence of impulsive channel noise,'' \emph{IEEE
  Transactions on Signal Processing}, vol.~63, no.~8, pp. 2118--2129, 2015.

\bibitem{Kar2009imc}
S.~Kar and J.~M.~F. Moura, ``Distributed consensus algorithms in sensor
  networks with imperfect communication: Link failures and channel noise,''
  \emph{IEEE Transactions on Signal Processing}, vol.~57, no.~1, pp. 355--369,
  2009.

\bibitem{Koloskova2019spare}
A.~Koloskova, S.~Stich, and M.~Jaggi, ``Decentralized stochastic optimization
  and gossip algorithms with compressed communication,'' in \emph{Proceedings
  of the 36th International Conference on Machine Learning}, ser. Proceedings
  of Machine Learning Research, K.~Chaudhuri and R.~Salakhutdinov, Eds.,
  vol.~97.\hskip 1em plus 0.5em minus 0.4em\relax PMLR, 09--15 Jun 2019, pp.
  3478--3487.

\bibitem{Wang2022Tailoring}
Y.~Wang and A.~Nedić, ``Tailoring gradient methods for differentially-private
  distributed optimization,'' \emph{IEEE Transactions on Automatic Control},
  pp. 1--16, 2023.

\bibitem{Wang2017DPL}
Y.~Wang, Z.~Huang, S.~Mitra, and G.~E. Dullerud, ``Differential privacy in
  linear distributed control systems: Entropy minimizing mechanisms and
  performance tradeoffs,'' \emph{IEEE Transactions on Control of Network
  Systems}, vol.~4, no.~1, pp. 118--130, 2017.

\bibitem{Sri2011async}
K.~Srivastava and A.~Nedic, ``Distributed asynchronous constrained stochastic
  optimization,'' \emph{IEEE Journal of Selected Topics in Signal Processing},
  vol.~5, no.~4, pp. 772--790, 2011.

\bibitem{Lei2018}
J.~Lei, H.~F. Chen, and H.~T. Fang, ``Asymptotic properties of primal-dual
  algorithm for distributed stochastic optimization over random networks with
  imperfect communications,'' \emph{SIAM Journal on Control and Optimization},
  vol.~56, no.~3, pp. 2159--2188, 2018.

\bibitem{Zhang2019Sign}
J.~Zhang, K.~You, and T.~Başar, ``Distributed discrete-time optimization in
  multiagent networks using only sign of relative state,'' \emph{IEEE
  Transactions on Automatic Control}, vol.~64, no.~6, pp. 2352--2367, 2019.

\bibitem{Doan2021Quanti}
T.~T. Doan, S.~T. Maguluri, and J.~Romberg, ``Convergence rates of distributed
  gradient methods under random quantization: A stochastic approximation
  approach,'' \emph{IEEE Transactions on Automatic Control}, vol.~66, no.~10,
  pp. 4469--4484, 2021.

\bibitem{Xin2018linear}
R.~Xin and U.~A. Khan, ``A linear algorithm for optimization over directed
  graphs with geometric convergence,'' \emph{IEEE Control Systems Letters},
  vol.~2, no.~3, pp. 315--320, 2018.

\bibitem{pu2018push}
S.~Pu, W.~Shi, J.~Xu, and A.~Nedić, ``A push-pull gradient method for
  distributed optimization in networks,'' in \emph{2018 IEEE Conference on
  Decision and Control (CDC)}, 2018, pp. 3385--3390.

\bibitem{pu2020robust}
S.~Pu, ``A robust gradient tracking method for distributed optimization over
  directed networks,'' in \emph{2020 59th IEEE Conference on Decision and
  Control (CDC)}, 2020, pp. 2335--2341.

\bibitem{wangGT2022}
Y.~Wang and T.~Başar, ``Gradient-tracking based distributed optimization with
  guaranteed optimality under noisy information sharing,'' \emph{IEEE
  Transactions on Automatic Control}, pp. 1--16, 2022.

\bibitem{Chen2022Priv}
X.~Chen, L.~Huang, L.~He, S.~Dey, and L.~Shi, ``A differential private method
  for distributed optimization in directed networks via state decomposition,''
  \emph{arXiv preprint arXiv:2107.04370}, 2021.

\bibitem{Ashok2019Momentum}
A.~Cutkosky and F.~Orabona, ``Momentum-based variance reduction in non-convex
  sgd,'' in \emph{Advances in Neural Information Processing Systems},
  vol.~32.\hskip 1em plus 0.5em minus 0.4em\relax Curran Associates, Inc.,
  2019.

\bibitem{polyak1987Introduction}
B.~T. Polyak, \emph{Introduction to Optimization}.\hskip 1em plus 0.5em minus
  0.4em\relax NY: Optimization Software, 1987.

\bibitem{Song2021CompressedGT}
Z.~Song, L.~Shi, S.~Pu, and M.~Yan, ``Compressed gradient tracking for
  decentralized optimization over general directed networks,'' \emph{arXiv
  preprint arXiv:2106.07243}, 2021.

\bibitem{zhao2020asymptotic}
S.~Zhao, X.~Chen, and Y.~Liu, ``Asymptotic properties of dual averaging
  algorithm for constrained distributed stochastic optimization,'' \emph{arXiv
  preprint arXiv:2009.02740}, 2020.

\end{thebibliography}

\end{document}